\documentclass[reqno, 11pt]{amsart}
 \usepackage{amsmath}
 \usepackage{amssymb}
 \usepackage{amsthm}
 \usepackage{times}
 \usepackage{latexsym}
 \usepackage[mathscr]{eucal}
 \usepackage{tikz-cd}

 \numberwithin{equation}{section}
 
 \newtheorem{theorem}{Theorem}[section]
  
 \newtheorem{proposition}[theorem]{Proposition}
 
 \newtheorem{corollary}[theorem]{Corollary}
 \newtheorem{remark}[theorem]{Remark}

 \newtheorem{definition}[theorem]{Definition}
 
 \newtheorem{nota}[theorem]{Notation}
 \def\slash#1{\setbox0=\hbox{$#1$}#1\hskip-\wd0\hbox to\wd0{\hss\sl/\/\hss}}

 \title[Ricci Curvature and Betti Numbers of Hessian Manifolds]{Ricci Curvature and Betti Numbers of Hessian Manifolds}
 
 \author[Emmanuel Gnandi,St\'ephane Puechmorel ]{Emmanuel Gnandi*,  St\'ephane Puechmorel**}

 \newcommand{\acr}{\newline\indent}
 
 \address{\llap{*\,}INSA de Toulouse\acr
 	Département de G\'enie Mathématique\acr
 	Université de Toulouse\acr
 	135 avenue de Rangueil\acr
 	31077 Toulouse cedex 4\acr
 	 France} 
 \email{kpanteemmanuel@gmail.com, gnandi@insa-toulouse.fr}
 \thanks{}
 
 \address{\llap{**\,}ENAC\acr 
 	Laboratoire ENAC\acr
 	Université de Toulouse\acr
 	7 Avenue Edouard Belin,\acr 
 	Toulouse, 31055}  
 \email{stephane.puechmorel@enac.fr}  
 \thanks{}  
 
 
\begin{document}
 
 \begin{abstract}   
We study Ricci curvature properties of Hessian metrics on the leaves of the codimension-one foliation $\mathcal{F}_\omega = \ker\,\omega$ generated by the first Koszul form $\omega$ of a closed oriented Hessian manifold. Our main 
result reveals a striking rigidity phenomenon: non-negative Ricci curvature on a 
single leaf of $\mathcal{F}_\omega$ compels the Hessian metric to be flat, 
yields sharp bounds on the first Betti number in terms of the dimension of the Hessian manifold and the 
topology of the leaves. This rigidity also shows that Koszul-type and radiant affine manifolds admit no leaf carrying non-negative Ricci curvature, reflecting a fundamental incompatibility between affine hyperbolicity and leafwise curvature positivity. In dimension three, we obtain a complete classification of the underlying manifold, extended to the non-orientable setting via the orientation double cover.
 \end{abstract} 
 	% Keywords 

 	\maketitle

 \section{Introduction}
 	
Hessian geometry sits at the confluence of several deep mathematical theories, bridging
differential geometry, information geometry, K\"{a}hlerian geometry, and mathematical physics.
A Riemannian metric on a locally flat manifold is called a Hessian metric if it can
be locally expressed as the Hessian of a smooth function with respect to affine coordinate
systems. A pair consisting of a flat structure and a Hessian metric is called a
Hessian structure, and a manifold equipped with such a structure is referred to as a
Hessian manifold. Canonical examples include regular convex cones and the space of
all positive definite real symmetric matrices.\\
A Hessian manifold $(M, g, \nabla)$ is simultaneously a statistical manifold in the sense of Amari~\cite{amari2000}, a locally flat manifold endowed with a compatible Riemannian
structure, and a natural framework for convex duality. This threefold nature makes Hessian
geometry one of the most fertile settings in modern geometric analysis.

A central structural result, due to Shima~\cite{shima2007}, asserts that the tangent bundle
$TM$ of any Hessian manifold carries a canonical K\"{a}hlerian metric, constructed from the
Hessian structure via the Dombrowski construction~\cite{Dombrowski1962}. This deep connection
with K\"{a}hlerian geometry reveals that Hessian manifolds are, in a precise sense, the real
analogues of K\"{a}hler manifolds: just as a K\"{a}hler metric is locally the complex Hessian
of a K\"{a}hler potential, a Hessian metric is locally the real Hessian of a convex potential
function~\cite{shima2007}. For this reason, S.\,Y.\ Cheng and S.\,T.\ Yau referred to
Hessian metrics as affine K\"{a}hler metrics. This connection is further formalized
through the notion of special K\"{a}hlerian geometry, in which Hessian structures arise
naturally~\cite{freed1999}. In~\cite{gnandi2025}, the author further confirms this deep
connection by proving that any compact orientable three-dimensional Hessian manifold is either
the Hantzsche-Wendt manifold or a K\"{a}hler mapping torus.\\
In information geometry, Hessian manifolds play a fundamental role. Statistical manifolds
arising from exponential families provide canonical examples: the Fisher information metric
coincides with the Hessian of the log-partition function, and the Amari--Chentsov dual affine
connections $\nabla^{(e)}$(exponential connection (e-connection)) and $\nabla^{(m)}$(the mixture connection (m-connection)) correspond precisely to the flat connections
$\nabla$ and $\nabla^{*}$ of the underlying Hessian structure~\cite{amari2000}.\\
Beyond these foundational connections, Hessian manifolds appear naturally in a remarkably
diverse range of mathematical and applied contexts. In algebraic geometry and symplectic geometry, they arise in the
study of Delzant polytopes, toric symplectic manifolds, and toric K\"{a}hler
manifolds~\cite{Fujita2024}, real Monge–Ampère equations~\cite{PuechmorelTo2023}, as well as in mirror
symmetry~\cite{zhang2020}. In mathematical physics, Souriau's geometric formulation of
thermodynamics on Lie groups~\cite{souriau1997} relies essentially on Hessian structures, a
perspective further developed by Barbaresco~\cite{barbaresco2015} in the framework of
Koszul--Vinberg geometry.\\
Armstrong and Amari~\cite{amari2014curvature} established that the Pontryagin forms of any
Hessian metric vanish identically. Moreover, compact manifolds with finite fundamental group
cannot support Hessian structures\cite{ay2002dually}, showing that the existence of such a
structure imposes strong topological constraints on the underlying manifold. In this direction,
Shima~\cite{shima1981hessian} proved that the universal affine covering of any compact
Hessian manifold is necessarily a convex domain, from which he derived the nonexistence of
Hessian metrics on Hopf manifolds $(\mathbb{S}^{m-1} \times \mathbb{S}^{1})$ for $m > 1$.\\

The aim of this paper is to study the Ricci curvature of Hessian metrics on the leaves of the
foliation $\mathcal{F}_\omega = \ker\,\omega$ defined by the first Koszul form of a closed
Hessian manifold $(M, g, \nabla)$. We show that non-negative Ricci curvature on a single leaf
forces the metric to be flat, yields sharp bounds on the first Betti number, and rules out
such leaves on radiant and hyperbolic affine manifolds. In dimension three, we derive a
complete classification of the underlying manifold.\\

The paper is organized as follows. Section~\ref{sec:pre} reviews the necessary background on
Hessian manifolds, affine hyperbolicity in the sense of Koszul, and $KV$-coho\-mology, a tool that gives a characterization of Hessian manifolds of Koszul type.
Section~\ref{sec:main} is devoted to the statement and proof of the main theorem, together
with the corollaries that can be derived from it.

\section{Preliminaries}\label{sec:pre}
%------------------------------------------------------------------
The first part of this section is devoted to Hessian manifolds and hyperbolic gauge structure sometimes called hyperbolic manifold in the sense of Koszul, not to be confused with an hyperbolic manifold. The second part introduced a cohomological tool that can be used to prove that a metric is Hessian. Since both are not so common, and for the sake of completeness, we recall the basic definitions and properties.
\subsection{Hessian manifolds}
\begin{definition}
\label{def:gauge_structure}
A gauge structure is a couple $\left( M, \nabla\right)$ where $M$ is a smooth manifold and $\nabla$ a Koszul connection on $TM.$
\end{definition}

\begin{definition}
A locally flat manifold (also called an affinely flat manifold) is a 
gauge structure $(M, \nabla)$, where $\nabla$ is a torsion-free connection whose curvature 
tensor vanishes identically:
\[
  T^\nabla = 0, \qquad R^\nabla = 0.
\]
The vanishing of $R^\nabla$ ensures that around every point $p \in M$ there exist 
local coordinates $(x^1, \dots, x^n)$, called affine coordinates, in which 
the Christoffel symbols of $\nabla$ vanish identically. Such a connection $\nabla$ 
is called a flat connection.
\end{definition}

\begin{remark}
By a classical theorem of differential geometry, a manifold $M$ admits a locally 
flat connection if and only if it admits an affinely flat structure, i.e., an atlas 
whose transition maps are affine transformations of $\mathbb{R}^n$. In particular, 
every locally flat manifold is locally diffeomorphic to an open subset of 
$\mathbb{R}^n$ equipped with its standard flat connection.
\end{remark}

The interplay between the flat connection $\nabla$ and the Riemannian metric $g$ 
is encoded by the difference tensor $s = D - \nabla$, where $D$ denotes 
the Levi-Civita connection of $(M, g)$. Since both $D$ and $\nabla$ are 
torsion-free, $s$ is a symmetric $(1,2)$-tensor:
\[
  s_X Y = s_Y X \qquad \text{for all } X, Y \in \mathfrak{X}(M).
\]
In affine coordinates $(x^i)$, the components $s^i_{jk}$ of $s$ coincide with the 
Christoffel symbols $\Gamma^i_{jk}$ of $D$, since $\nabla$ has vanishing 
Christoffel symbols in these coordinates. The tensor $s$ thus measures the 
deviation of $g$ from being compatible with $\nabla$, and plays a central role in 
the structure theory of Hessian manifolds.

\begin{definition}[\cite{shima1997, shima2007}]
A Riemannian metric $g$ on a locally flat manifold $(M, \nabla)$ is called a 
Hessian metric if every point of $M$ admits a neighbourhood $U$ and a 
smooth strictly convex function $\phi \in C^\infty(U)$, called a 
Hessian potential, such that
\[
  g = \nabla^2 \phi, \qquad \text{i.e.,} \qquad
  g_{ij} = \frac{\partial^2 \phi}{\partial x^i \partial x^j},
\]
where $(x^1, \dots, x^n)$ is an affine coordinate system with respect to $\nabla$.
This is the real analogue of the K\"ahler condition $g = i\partial\bar\partial\psi$.
A locally flat manifold $(M, \nabla)$ endowed with a Hessian metric $g$ is called a 
Hessian manifold, denoted $(M, g, \nabla)$.
\end{definition}

The following proposition, due to Shima, provides several equivalent 
characterizations of Hessian metrics. Condition~(2) is a Codazzi-type 
equation relating $\nabla$ and $g$, conditions~(3) and~(5) are its coordinate 
expressions, and condition~(4) asserts the self-adjointness of $s_X$ with respect 
to $g$.

\begin{proposition}[\cite{shima1997}]
Let $(M, \nabla)$ be a locally flat manifold and $g$ a Riemannian metric on $M$.
The following conditions are equivalent:
\begin{enumerate}
  \item $g$ is a Hessian metric;
  \item $(\nabla_X g)(Y,Z) = (\nabla_Y g)(X,Z)$ for all 
        $X, Y, Z \in \mathfrak{X}(M)$;
  \item In affine coordinates $(x^i)$, the components $g_{ij}$ satisfy
        \[
          \frac{\partial g_{ij}}{\partial x^k} = \frac{\partial g_{kj}}{\partial x^i};
        \]
  \item The tensor $s_X$ is self-adjoint with respect to $g$ for every 
        $X \in \mathfrak{X}(M)$:
        \[
          g(s_X Y, Z) = g(Y, s_X Z) \qquad 
          \text{for all } Y, Z \in \mathfrak{X}(M);
        \]
  \item $s_{ijk} = s_{jik}$, where $s_{ijk} = g_{il}\, s^l_{jk}$.
\end{enumerate}
\end{proposition}

The Hessian structure $(\nabla, g)$ on an orientable manifold gives rise to a 
canonical closed $1$-form, introduced by Koszul in the context of convex 
homogeneous domains~\cite{koszul1961}. Koszul observed that the flat 
connection $\nabla$ acts on the volume form $\mu_g$ of $g$ by a multiplicative 
scalar, thereby defining a $1$-form. Following Shima~\cite{shima2007}, 
this form is called the first Koszul form of $(\nabla, g)$.

\begin{definition}[\cite{koszul1961, shima1997, PuechmorelTo2023}]
Let $(M, g, \nabla)$ be an orientable Hessian manifold, and let $\mu_g$ denote 
the Riemannian volume form of $g$. The first Koszul form $\omega$ of the 
Hessian structure $(\nabla, g)$ is the $1$-form defined by
\[
  \nabla_X \mu_g = \omega(X)\, \mu_g,
  \qquad X \in \mathfrak{X}(M).
\]
\end{definition}

A direct computation in affine coordinates yields the explicit expressions
\[
  \omega(X) = \operatorname{tr}(s_X), \qquad
  \omega_i
  = \frac{1}{2}\,\frac{\partial \log \det(g_{pq})}{\partial x^i}
  = s^k_{ki}.
\]
In particular, $\omega$ is a closed $1$-form~\cite{shima2007}, and its 
cohomology class $[\omega] \in H^1(M, \mathbb{R})$ is a global invariant of the 
Hessian structure. When $(M, g, \nabla)$ arises from a convex cone via the 
Koszul-Vinberg construction, $\omega$ coincides with the logarithmic derivative 
of the characteristic function of the cone.

By a theorem of Shima and Yagi~\cite[Theorem~4.1]{shima1997}, on a 
compact orientable Hessian manifold the first Koszul form $\omega$ satisfies
\[
  D\omega = 0,
\]
so that its pointwise norm $\|\omega\|_g$ is constant on $M$. Consequently, 
exactly one of the following two cases occurs: either $\omega \equiv 0$, or 
$\omega$ is everywhere non-vanishing. An example of a non-compact Hessian manifold 
satisfying $D\omega \ne 0$ \cite[p.~282]{shima1997}.

In the case $\omega \equiv 0$, by~\cite[Theorem~4.2]{shima1997} the 
Levi-Civita and flat connections coincide, $D = \nabla$, so $(M, g)$ is a flat 
Riemannian manifold. By the Bieberbach 
theorems~\cite{bieberbach1911, wolf2023}, every compact 
flat Riemannian manifold is finitely covered by a flat torus. 
In particular, the first Betti number satisfies
\[
  0 \leq b_1(M) \leq dim(M).
\]
Moreover, by a result of Shima~\cite[Corollary~6]{shima1980}, every 
compact connected homogeneous Hessian manifold of dimension $d$ is isometric to 
a Euclidean torus, in which case $b_1(M) = d$.

The Ricci tensor $\mathrm{Ric}^g$ of the Levi-Civita connection $D$ is expressed 
in affine coordinates $(x^i)$ by
\[
  R^g_{jk} = s^{r}_{sj}\, s^{s}_{rk} - s^{r}_{rs}\, s^{s}_{jk},
\]
where $s^i_{jk} = \Gamma^i_{jk}$ are the Christoffel symbols of $D$ in affine 
coordinates. Curvature properties of Hessian metrics have been studied 
in~\cite{shima2007, totaro2004, yilmaz2008}.
\subsection{Hyperbolic connections}
\begin{nota}
If $\gamma \colon [0,1] \to M$ is a path, the parallel transport $T_{\gamma(0)}M \to T_{\gamma(t)}M, t \in [0,1]$ along $\gamma$ 
with respect to an affine connection $\nabla$
will be denoted by $\Pi^\nabla_t(\gamma),$ or simply  $\Pi_t(\gamma)$ 
if there is no risk of confusion.
\end{nota}
\begin{definition}
\label{def:developing_map}
Let $\left(M, \nabla \right)$ be a gauge structure on a connected manifold $M$ and let $x_0 \in M.$ Let
$\gamma \colon [0,1] \to M$ be a path such that $x_0 = \gamma(0)$ Its development in $T_{x_0}M$ is the path:
\begin{equation*}
    \tau(\gamma) \colon t \in [0,1] \mapsto \int_0^t \Pi^{-1}_u(\gamma) \gamma^\prime(u) du.
\end{equation*}
\end{definition}
\begin{proposition}
  \label{prop:affine_dev}
Let $\gamma_1, \gamma_0 \colon [0,1]$ be paths such that $\gamma_1(0)=\gamma_0(1)$ and let $\gamma_1 \cdot \gamma_0$ be their composition. Then:
\begin{equation*}
    \begin{split}
        &\Pi_1(\gamma_1 \cdot \gamma_0) = \Pi_1(\gamma_1)\Pi_1(\gamma_0) \\
        &\tau(\gamma_1 \cdot \gamma_0)(1) = \Pi_1(\gamma_1)^{-1}\tau(\gamma_0)(1) + \tau(\gamma_1)(1). \\
    \end{split}
\end{equation*}
\end{proposition}
\begin{proof}
The first property is a direct consequence of the observation that, for all $v \in T_{\gamma(0)}M,$ $\Pi_t(\gamma)v$ is a solution to the first order differential equation $\nabla_{\gamma^\prime}\Pi_t v = 0.$
The second one is readily obtained by splitting the integral defining the development into a part
\end{proof}
In the sequel $\tau_1(\gamma)$ will stand for $\tau(\gamma)(1)$ and will be called the development of 
$\gamma$ at $x_0.$
When the gauge structure $(M,\nabla)$ is flat, i.e. $\nabla$ has vanishing curvature,
 the development depends only on the homotopy class of the path, 
 hence is well-defined as a mapping on $\tilde{M},$ where $\tilde{M}$ is a universal covering of $M$.
 The next proposition is a direct consequence of proposition \ref{prop:affine_dev}.
 \begin{proposition}
  \label{prop:affine_action}
 Let $(M, \nabla)$ be a locally flat gauge structure and $x_0 \in M$. Let $\pi_{x_0}$ be the fundamental group
 at $x_0$. Then the mapping $\gamma \in \pi_{x_0} \mapsto \left( \Phi_1(\gamma), \tau_1(\gamma) \right)$ 
 is an affine representation of $\pi_{x_0}$ in $T_{x_0}M.$
 \end{proposition}

\begin{definition}
\label{def:hyperbolic_koszul}
A locally flat gauge structure $\left( M, \nabla \right)$ is said to be 
hyperbolic in the sense of Koszul \cite{koszul1968} if the universal covering $\tilde{M}$ provided 
with the induced connection $\nabla$ is isomorphic to an open convex subset of $\mathbb{R}^n$ containing non line.
\end{definition}
If $\left( M, \nabla \right)$ is hyperbolic, then the domain $E^\nabla_{x_0} \subset T_{x_0}$ of the 
exponential $\exp^\nabla$ is a convex containing no line and 
\begin{equation*}
 \begin{tikzcd}
E^\nabla_{x_0} \subset T_{x_0} \ar[r,"\exp_{x_0}^\nabla"] & M
\end{tikzcd}   
\end{equation*}
is a universal covering of $M.$ It is invariant by the affine action of proposition \ref{prop:affine_action},
thus, if $E^\nabla_{x_0}$ is a cone, the action is in fact linear. 
If $M$ is compact, then the converse proposition holds, that is if the action is linear, then $E^\nabla_{x_0}$ 
is a cone.
When the above properties hold, then there exists an affine vector field $X$ and a Riemannian metric $g$ such that 
$\mathcal{L}_X g = 0,$ where $\mathcal{L}$ is the Lie derivative. When $M$ is hyperbolic and compact, 
$X$ is unique.

\begin{theorem}[Koszul~\cite{koszul1968}]
\label{thm:Hyper}
Let $(M,\nabla)$ be a locally flat manifold.  
If $(M,\nabla)$ is hyperbolic, then there exists a de Rham closed 
differential $1$-form $\beta$ on $M$ such that its covariant derivative 
$\nabla \beta$ is positive definite.  
If $M$ is compact, this condition is also sufficient.
\end{theorem}

Hyperbolic manifolds $(M,\nabla)$ are intimately connected to Hessian geometry. 
Given a closed $1$-form $\beta$, its covariant derivative $\nabla\beta$ yields a symmetric $2$-tensor which, 
by Koszul's characterization of hyperbolicity, is positive definite and thus endows $M$ with a Riemannian metric. 
The local exactness of closed forms allows one to write $\beta = d\varphi$ for some smooth function $\varphi$, 
under which $\nabla\beta$ reduces to the Hessian $d^2\varphi$.

\subsection{Koszul-Vinberg cohomology}

Let $\left( M, \nabla \right)$ be a locally flat gauge structure. It defines a product on the space $\chi(M)$ of smooth sections of $TM$ by:
\begin{equation}
\label{eq:product}
    X \cdot Y = \nabla_X Y, \, X,Y \in \chi(M).
\end{equation}
The commutator $X\cdot Y - Y \cdot X$ is the usual Lie bracket and the associator is given by the relation:
\begin{equation}
    \label{eq:associator_nabla}
    \left( X, Y, Z \right) = \nabla_{\nabla_X Y} Z - \nabla_X \nabla_Y Z = - \nabla^2_{X,Y}Z, \, X, Y, Z \in \chi\left(M\right).
\end{equation}
for an arbitrary connection with curvature $R^\nabla,$ the next proposition holds:
\begin{proposition}
\label{prop:kv_defect}
For any $X,Y,Z \in \chi\left(M\right):$
\begin{equation*}
\label{eq:kv_defect}
    \left( X, Y, Z \right) - \left( Y, X, Z \right) = - R^\nabla\left( X, Y \right) Z,
\end{equation*}
\end{proposition}
When $R^\nabla$ vanishes, the product \ref{eq:product} provides $\chi(M)$ with a Koszul-Vinberg, (or pre-Lie) algebra structure. To any locally flat gauge structure, there is thus an associated Koszul-Vinberg (KV in short) algebra. The next proposition is very classical 
\begin{proposition}
\label{prop:subjacent_lie}
Let $(A,\cdot)$ be a KV algebra. The bracket:
\begin{equation*}
\left [ X, Y \right ] = X \cdot Y - Y \cdot X
\end{equation*}
define a Lie algebra structure on $A,$ called the subjacent Lie algebra.
\end{proposition}
\begin{proof}
The Jacobi identity comes directly from the next expression:
\begin{equation*}
    \begin{split}
        & [X,[Y,Z]] + [Z,[X,Y]] + [Y,[Z,X]] = \\
        & \left( X, Y , Z \right) - \left( X, Z , Y \right) + \left( Y, Z , X \right) 
         +\left( Z, X , Y \right) - \left( Z, Y , X \right) - \left( Y, X , Z \right)
    \end{split}
\end{equation*}
\end{proof}
\begin{remark}
When the connection $\nabla$ is torsion-free, the subjacent Lie algebra 
is $\chi(M)$, the Lie algebra of vector fields.
\end{remark}
\begin{definition}[\cite{bai2008}]
Let $A$ be a KV algebra. A KV-module $V$ is a vector space equipped with respective left and right actions:
\[
X \triangleleft v, v \triangleright X, \, X \in A, v \in V
\]
such that $\triangleleft$ is a Lie algebra representation from the subjacent Lie algebra of $A$ to  $\mathfrak{gl}(V)$ and $\triangleright$  satisfies:
\[
\left( X \triangleleft v \right) \triangleright Y - X \triangleleft \left( 
v \triangleright Y \right) = \left( v \triangleright X \right) \triangleright Y - v \triangleright \left( X \cdot Y \right), \, X,Y \in A, \, v \in V.
\]
\end{definition}
\begin{definition}
    \label{def:kv_complex}
    Let $A$ be a KV algebra and $V$ a $A$-module.For any positive integer $q$, $C^q\left( A, V \right) = \hom_{\mathbb{R}} \left( \otimes^q A, V \right)$ has a natural $KV$-module structure with actions:
    \begin{equation}
        \label{eq:kv_cq}
        \begin{split}
            & \left( a \triangleleft f \right)\left( a_1 \otimes \dots \otimes a_q \right) = a \triangleleft f\left( a_1 \otimes \dots \otimes a_q \right) - \sum_{i=1}^q f\left( 
                a_1 \otimes \dots \otimes \left(a \circ a_i\right) \otimes \dots a_q
             \right) \\
             & \left( f \triangleright a \right)\left( a_1 \otimes \dots \otimes a_q \right) = f\left( a_1 \otimes \dots \otimes a_q \right) \triangleright a.
        \end{split}
    \end{equation}
\end{definition}
\begin{definition}
\label{def:kv_face_insert}
Let $q$ be a positive integer. The face (resp. insertion) operator is defined by:
\begin{equation}
    \label{eq:kv_face}
    \partial_i \colon a_1 \otimes \dots \otimes a_q = a_1 \otimes \dots \otimes \hat{a}_i \otimes \dots \otimes a_q
\end{equation}
resp.
\begin{equation}
    \label{eq:kv_insert}
    \mathcal{I}_i(a) \colon a_1 \otimes \dots \otimes a_q = 
    \begin{cases}
        a_1 \otimes \dots  a_{i-1} \otimes a \otimes a_i \otimes \dots \otimes a_q, \, a \in A, \, i=2 \dots q\\
        a \otimes a_1 \otimes \dots \otimes a_q, \, a \in A, i=1 \\
        a_1 \otimes \dots \otimes a_q \otimes a, \, a \in A, i=q+1 \\
    \end{cases}
\end{equation}
\end{definition}
\begin{definition}
    \label{def:kv_coboundary}
    The coboundary operator $\delta^q \colon C^q \left( A,V \right) \to C^{q+1}\left( A,V \right)$ is defined by \cite{Boyom2002}:
    \begin{equation}
        \label{eq:coboundary}
        \delta^q f = \sum_{j=1 \dots q} \left( -1 \right)^j \left[ a_j  \triangleleft \left(f \partial_j\right) +  f\left( \mathcal{I}_q(a_j)\partial_{q}\partial_j\right)\triangleright a_{q+1}  \right]
    \end{equation}
\end{definition}
The fact that $\delta^2=0$ can be established by brute force \cite{Boyom2002}, but is a consequence of semi-simplicial identities between $\triangleleft$ and $\triangleright$.

In the present work, the previous construction will be applied to $C^\infty(M),$ and $\left(\chi(M),\nabla\right)$ with left action:
\[
X \triangleleft f = X(f)
\]
and trivial right action. 
\begin{definition}
Let $(M,\nabla)$ be a locally flat gauge structure. The \emph{Koszul--Vinberg complex} is defined as
\[
\bigl(C(\nabla), \delta \bigr) = \left( \, C(\nabla) = \bigoplus_{j \geq 0} C^j(\nabla),
\quad \delta_j : C^j(\nabla) \to C^{j+1}(\nabla) \,\right),
\]
where
\[
C^0(\nabla) = \{ f \in C^\infty(M) \mid \nabla^2 f = 0 \}, 
\qquad
C^j(\nabla) = \mathrm{Hom}_{\mathbb{R}}\!\left( \otimes^j \mathcal{X}(M), \, C^\infty(M)\right), 
\quad j \geq 1.
\]

The coboundary operator $\delta_j$ is defined as in \ref{def:kv_coboundary}, except for the $0$-th term:
\begin{equation}
\delta_0 f = df, \quad \forall f \in C^0(\nabla),
\end{equation}
and for $j \ge 1$,

The \emph{Total Koszul--Vinberg cohomology} is then the graded vector space
\begin{equation}
H^j_{\tau}(\nabla) = \ker \delta_j \, / \, \mathrm{im}\,\delta_{j-1}.
\end{equation}
\end{definition}

\begin{proposition}
\label{prop:kv_hessian}
Let $(M, g, \nabla)$ be a Hessian manifold. Then;
\begin{itemize}
    \item The Hessian metric $g$ is a cocycle of the scalar KV complex, and $[g] \in H^{2}_{\tau}(\nabla)$. 
    \item If $[g]=0,$ the Hessian manifold is of Koszul type, that is $g = \nabla \beta$ for $\beta$ a closed $1$-form.
\end{itemize}
\end{proposition}
Proposition \ref{prop:kv_hessian} is of pure cohomological nature, hence allows the use of tools like spectral sequences \cite{boyom2024}. This fact motivates the corollary \ref{cor:noo-nonneg-ricci} stated in the next section. 
 
\section{Main results}\label{sec:main}
Our main result can be stated as follows.
\begin{theorem}\label{thm:main}
Let $(M^d, g, \nabla)$ be a closed, oriented Hessian manifold of dimension $d$, with
nowhere vanishing first Koszul form $\omega$. The distribution $\mathcal{F}_\omega = \ker\omega$ defines a totally geodesic 
codimension-one foliation on $M^d$. Suppose there exists a leaf $L$ of $\mathcal{F}_\omega$ such that
 $\operatorname{Ric}^g|_{TL\times TL} \ge 0$. Then the Hessian metric $g$ is flat and
\[
1 \le b_1(M^d) \le d \quad (\text{resp. } 2 \le b_1(M^d) \le d,\; b_2(M^d) \ge 1)
\]
whenever the leaves of $\mathcal{F}_\omega$ are compact (resp. dense in $M^d$). 
Moreover, $b_1(M^d) = d$ if and only if $M^d$ is isometric to the flat torus $\mathbb{T}^d$. If additionally $\operatorname{Ric}^g|_{TL\times TL} > 0$ and $d \in \{3,4\}$, then $b_1(M^d) \in \{1, d\}$.
\end{theorem}

\begin{proof}
Let $(M^d, g, \nabla)$ be a closed, oriented Hessian manifold of dimension $d$, with first Koszul form $\omega$. 
Define the vector field $V$ by $\omega = g(V, \cdot)$.
 By \cite[Theorem~4.1]{shima1997}, $\omega$ satisfies $D\omega = 0$, hence $d\omega = 0$ and $\|\omega\|_g = k$ 
 for some non-zero constant $k$. The distribution $\ker\omega$ is a codimension-one subbundle of $TM$.
  From $D\omega = 0$, a direct computation gives $X(\omega(Y)) = \omega(D_X Y)$ for all $X, Y \in \Gamma(TM)$.
   In particular, if $Y \in \Gamma(\ker\omega)$, then $D_X Y \in \Gamma(\ker\omega)$ for every $X \in \Gamma(TM)$,
  and integrability of $\mathcal{F}_\omega = \ker\omega$ follows from \cite[Lemma~1.5.1]{furness1972}.
   Moreover, $D_X Y \in \Gamma(\ker\omega)$ for all $X \in \Gamma(TM)$ implies
\[
  g(D_{Y'} Y,\, Z) = 0
  \qquad
  \forall\, Y, Y' \in \Gamma(\ker\omega),\quad \forall\, Z \in \Gamma((\ker\omega)^\perp).
\]
By \cite[Theorem~5.23]{tondeur2012}, this condition is equivalent to $\mathcal{F}_\omega$ being a totally
 geodesic codimension-one foliation on $M^d$. 
Consequently, by \cite{blumenthal1983},
\[
M^d \simeq \frac{\widetilde{L}\times \mathbb{R}}{\pi_1(M^d)},
\]
where $\widetilde{L}$ is the universal cover of the leaves of $\mathcal{F}_\omega$.
 By \cite[p.~111]{reeb1952}, the leaves of $\mathcal{F}_\omega$ are all homeomorphic, 
 and are either all compact or all dense in $M^d$. When the leaves are compact, \cite{godbillon1991,reeb1952}
  gives that $M^d$ is the total space of a fibration $p \colon M^d \to \mathbb{S}^1$ with fibre $L$, 
  and $\mathcal{F}_\omega$ is the fibre foliation $\{p^{-1}(\theta) \mid \theta \in \mathbb{S}^1\}$.
   Hence $\pi_1(M^d) \cong \pi_1(L) \rtimes \mathbb{Z}$, so $H_1(M^d;\mathbb{Z})$ has rank at least $1$ and $b_1(M^d) \ge 1$. 
   When the leaves are dense, $M^d$ compact implies $\pi_1(M^d)$ is finitely generated. By \cite[pp.~45--46]{godbillon1991}, 
   the quotient
\[
  \frac{\pi_1(M^d)}{\pi_1(L)} \simeq K
\]
is finitely generated, torsion-free, and abelian. Since the leaves are dense, \cite[p.~46]{godbillon1991}
 implies $K$ is acyclic, hence $K \simeq \mathbb{Z}^m$ for some $m \ge 2$. Thus we have a short exact sequence
\[
  1 \longrightarrow \pi_1(L) \longrightarrow \pi_1(M^d)
  \longrightarrow \mathbb{Z}^m \longrightarrow 1, \qquad m \ge 2.
\]
Abelianisation is a right exact functor, so this induces
\[
  H_1(L;\mathbb{Z}) \xrightarrow{i_*} H_1(M^d;\mathbb{Z})
  \xrightarrow{p_*} \mathbb{Z}^m \longrightarrow 0,
\]
and therefore $m \le \operatorname{rank}H_1(M;\mathbb{Z}) = b_1(M^d)$, giving $b_1(M^d) \ge m \ge 2$. 

Since $D\omega = 0$ implies $DV = 0$, Chern's criterion \cite{chern1966} 
gives $b_2(M^d) \ge b_1(M^d) - 1$, hence $b_2(M^d) \ge 1$.

Using the orthogonal decomposition $TM^d = \mathbb{R}V \oplus \ker\omega$, 
any vector field $U \in \Gamma(TM^d)$ can be written as $U = fV + X$ with $f \in C^\infty(M^d)$ and 
$X \in \Gamma(\ker\omega)$. A direct computation gives
\[
  \operatorname{Ric}(U,U)
  = f^2\operatorname{Ric}(V,V) + 2f\operatorname{Ric}(X,V)
    + \operatorname{Ric}(X,X).
\]
Since $DV = 0$, the vector field $V$ is parallel, and the two cross terms vanish, leaving
\[
  \operatorname{Ric}(U,U) = \operatorname{Ric}(X,X).
\]
By assumption, $\operatorname{Ric}^g|_{TL \times TL} \ge 0$, so $\operatorname{Ric}(X,X) \ge 0$ 
for all $X \in \Gamma(\ker\omega)$, and consequently $\operatorname{Ric}^g \ge 0$ on all of $TM^d$. 
Bochner's theorem \cite{Bochner1946,yano1953} then yields $b_1(M^d) \le d$, and we deduce that $2 - \gamma \le b_1(M^d) \le d, \gamma = 1 \text{ if leaves are compact, and } \gamma = 0 \text{ if leaves are dense.}$
Moreover, when $b_1(M^d) = d$, Bochner's theorem forces $M^d$ to be isometric to the flat torus $\mathbb{T}^d$.

Since the universal cover of a compact Hessian manifold is a convex domain \cite{shima1981hessian},
 it is contractible, so $\pi_k(M^d) = 0$ for all $k \ge 2$ and $M$ is a $K(\pi_1(M^d),1)$-space. 
 The Cheeger--Gromoll splitting theorem \cite[Corollary~7.3.13]{petersen2006} then implies that $g$ is flat. 

Now assume additionally that $\operatorname{Ric}^g|_{TL \times TL} > 0$. In the compact leaves case, 
by \cite[p.~46]{godbillon1991}, $m = 1$, and the short exact sequence
\[
  1 \longrightarrow \ker(p_*)
  \longrightarrow H_1(M^d;\mathbb{Z})
  \longrightarrow \mathbb{Z}
  \longrightarrow 1
\]
gives $b_1(M^d) = \operatorname{rank}(\ker p_*) + 1$. Since
\[
  \operatorname{rank}(\ker p_*)
  = \operatorname{rank}\!\left(H_1(L;\mathbb{Z}) \big/ \ker(i_*)\right)
  \le \operatorname{rank}(H_1(L;\mathbb{Z})) = b_1(L),
\]
we obtain $b_1(M^d) \le b_1(L) + 1$. By Myers' theorem \cite{myers1941}, positive Ricci curvature
 implies $\pi_1(L)$ is finite, so $b_1(L) = 0$, and therefore $b_1(M^d) = 1$.

In the dense leaves case, assume $d = 4$. Since $M^d$ is compact, \cite[Exercise~10.4.28]{conlon2001}
 implies the leaves are complete. The leaves are three-dimensional complete non-compact (open) manifolds with positive Ricci curvature, 
 hence diffeomorphic to $\mathbb{R}^3$ \cite{shi1989} (Theorem 2.5), and thus these leaves are simply connected. 
 Therefore $\pi_1(M^d) \simeq \mathbb{Z}^m$ for some $m \ge 2$, so $M^d$ is a $K(\mathbb{Z}^m,1)$-space, homotopy-equivalent
  to $\mathbb{T}^m$ \cite[p.~49]{godbillon1991}. Since $M^d$ is orientable, $H_d(M^d;\mathbb{Z}) \simeq \mathbb{Z}$, 
  while $H_d(\mathbb{T}^m;\mathbb{Z}) \simeq \mathbb{Z}^{\binom{m}{d}}$, so $\binom{m}{d} = 1$, which gives $m = d$.
   Hence $M^d$ is homeomorphic to $\mathbb{T}^4$ and $b_1(M^d) = 4$. In the case $d = 3$, by the same argument,
    the leaves are open surfaces. The leaves are thus diffeomorphic to either $\mathbb{R}^2$, the M\"obius strip, 
    or $\mathbb{S}^1 \times \mathbb{R}$ \cite[p.~49]{godbillon1991}. Since $M^d$ is orientable, $\mathcal{F}_\omega$ is transversally 
    orientable, hence its leaves are orientable; the M\"obius strip is therefore excluded. The leaves are thus diffeomorphic to either
     $\mathbb{R}^2$ or $\mathbb{S}^1 \times \mathbb{R}$. Positive Ricci curvature on these surfaces implies $b_1(L) = 0$ by 
     \cite{anderson1990}, so the leaves are diffeomorphic to $\mathbb{R}^2$ and are simply connected. The same argument as in the $4$-dimensional case yields that $M^d$ is homeomorphic to $\mathbb{T}^3$ and $b_1(M^d) = 3$.

\medskip
In all cases, $b_1(M) \in \{1, d\}$, completing the proof.
\end{proof}

We can now state the following corollaries

\begin{corollary}\label{cor:no-nonneg-ricci}
Let $(M^d, g, \nabla)$ be a closed, oriented Hessian manifold. 
Suppose that either $(M^d, g, \nabla)$ is of Koszul type, or 
$(M^d, \nabla)$ is a radiant affine manifold. Then no leaf $L$ 
of the foliation $\mathcal{F}_\omega=\ker\,\omega,$
where $\omega$ denotes the first Koszul form of $(M^d, g, \nabla)$, 
carries non-negative Ricci curvature with respect to the induced 
metric $g|_L$.
\end{corollary}

\begin{proof}
\text{Case 1: $(M^d, g, \nabla)$ is of Koszul type.}

By definition, there exists a closed $1$-form $\omega$ on $M$ such that
$g = \nabla\omega$ \cite{shima2007}. By \cite{koszul1968, shima2007},
this implies that $(M, \nabla)$ is affine hyperbolic, and $M$ is
diffeomorphic to an orbit space
\[
M \;\cong\; \Gamma \backslash O,
\]
where $O \subset \mathbb{R}^d$ is an open convex domain containing no
complete affine line, and $\Gamma$ is a discrete subgroup of the affine
group $\mathrm{Aff}(d)$ acting properly and freely on $O$
(see Koszul~\cite{koszul1968} and Vey~\cite{vey1969}).

By \cite[Theorem~8.4]{shima2007}, there exists an affine coordinate
system $\{y^1, \ldots, y^d\}$ such that $y^i > 0$ on $O$ for all $i$,
and the tube domain $TO = \mathbb{R}^d + \sqrt{-1}\,O$ admits a Bergman
kernel $K = K(y^1, \ldots, y^d)$. The volume element on $O$ is given by
\[
\nu_O \;=\; \sqrt{K}\; dy^1 \wedge \cdots \wedge dy^d,
\]
and the first Koszul form on $O$ takes the explicit form
\[
\alpha \;=\; \frac{1}{2} \sum_{i} \frac{\partial \log K}{\partial y^i}\, dy^i.
\]
Denoting by $p : O \to M$ the covering projection, one has $p^*\omega = \alpha$
and $p^*\nu_g = \nu_O$. Moreover, a direct computation using the flatness
of the affine connection $\nabla^0$ (defined by $\nabla^0 \partial_i = 0$
and satisfying $dp(\nabla^0) = \nabla$) yields:
\[
\nabla^0 \nu_O \;=\; \alpha \otimes \nu_O.
\]
Since $\alpha$ is $\Gamma$-invariant, this relation descends to $M$ and gives
\[
\nabla \nu_g \;=\; \omega \otimes \nu_g,
\]
confirming that $\omega$ is indeed the first Koszul form of the Hessian
structure $(g, \nabla)$. In particular, since the domain $O$ is a proper
convex domain and the Bergman kernel $K$ is non-constant on $O$, the
form $\omega$ is nowhere vanishing on $M$.

Suppose for contradiction that some leaf $L$ of
$\mathcal{F}_\omega = \ker\,\omega$ satisfies
\[
\mathrm{Ric}^g\big|_{TL \times TL} \;\geq\; 0.
\]
Then by Theorem~\ref{thm:main}, the Hessian metric $g = \nabla\omega$ must
be flat, so that
\[
M^d \;\simeq\; \mathbb{R}^d / \Lambda,
\]
where $\Lambda \subset \mathbb{R}^d$ is a Bieberbach group acting freely
and properly discontinuously on $\mathbb{R}^d$, and the universal cover
of $M^d$ is $\mathbb{R}^d$. However, this contradicts the hyperbolicity
of $(M, \nabla)$, whose universal cover must be a proper convex domain
strictly contained in $\mathbb{R}^d$. Hence no leaf of $\mathcal{F}_\omega$
admits non-negative Ricci curvature with respect to $g|_L$.

\medskip

\text{Case 2: $(M^d, \nabla)$ is a radiant affine manifold.}

By \cite[Theorem~5.8]{Gnandi2026}, there exists a closed $1$-form $\Omega$
on $M$ such that $g = \nabla^*\Omega$, where $\nabla^*$ denotes the dual
connection of $\nabla$, defined by
\[
\nabla^* \;=\; 2D - \nabla,
\]
with $D$ the Levi-Civita connection of $g$. It is well known that
$(g, \nabla^*)$ is itself a Hessian structure on $M$, so that
$(M^d, g, \nabla^*)$ is a Hessian manifold of Koszul type. The conclusion
then follows immediately from Case~1 applied to $(M^d, g, \nabla^*)$.
\end{proof}

\begin{corollary}\label{cor:noo-nonneg-ricci}
Let $(M^d, g, \nabla)$ be a closed, oriented Hessian manifold. 
If the cohomology class $[g] \in H^{2}_{\tau}(\nabla)$ 
is trivial, then no leaf $L$ of the foliation
$\mathcal{F}_\omega \;=\; \ker\,\omega,$
where $\omega$ denotes the first Koszul form of $(M^d, g, \nabla)$,
carries non-negative Ricci curvature with respect to the induced 
metric $g|_L$.
\end{corollary}

\begin{proof}
By \cite[Corollary~2]{boyom2024}, the vanishing of the 
Koszul--Vinberg cohomology class 
$[g] \in H^{2}_{\tau}(\nabla)$ implies that 
$(M^d, g, \nabla)$ is a Hessian manifold of Koszul type. 
The conclusion then follows immediately from 
Corollary~\ref{cor:no-nonneg-ricci}.
\end{proof}

The following proposition gives a complete classification of closed, 
oriented Hessian $3$-manifolds satisfying the Ricci condition 
of Theorem~\ref{thm:main}.

\begin{proposition}\label{pro:dim3}
Under the same hypotheses as in Theorem~\ref{thm:main} with $d = 3$,
if $\operatorname{Ric}^g|_{TL \times TL} \geq 0$, then 
$b_1(M^3) \in \{1, 3\}$ and $M^3$ is diffeomorphic to either the 
flat torus $\mathbb{T}^3$, or a torus bundle 
$\mathbb{T}^2 \hookrightarrow M^3 \to \mathbb{S}^1$ whose monodromy 
$A \in \mathrm{SL}(2,\mathbb{Z})$ is periodic, i.e.\ conjugate in 
$\mathrm{GL}(2,\mathbb{Z})$ to one of
\[
\begin{pmatrix} 0 & 1 \\ -1 & 0 \end{pmatrix}, \qquad
\begin{pmatrix} -1 & 0 \\ 0 & -1 \end{pmatrix}, \qquad
\begin{pmatrix} 0 & -1 \\ 1 & 1 \end{pmatrix}, \qquad
\begin{pmatrix} -1 & -1 \\ 1 & 0 \end{pmatrix}.
\]
\end{proposition}

\begin{proof}
By Theorem~\ref{thm:main}, the condition 
$\operatorname{Ric}^g|_{TL \times TL} \geq 0$ implies that 
$b_1(M^3) \in \{1, 2, 3\}$. By 
\cite[Theorem~3.5]{gnandi2025}, the case $b_1(M^3) = 2$ 
is excluded, so that $b_1(M^3) \in \{1, 3\}$.

Moreover, Theorem~\ref{thm:main} implies that the Hessian metric 
$g$ is flat, so that $(M^3, g)$ is a closed, oriented, flat 
Riemannian $3$-manifold. By Bieberbach's theorem, $M^3$ is finitely 
covered by $\mathbb{T}^3$. The classification of closed oriented 
flat $3$-manifolds (see~\cite{Wolf1974, bieberbach1912}) 
shows that $M^3$ is either diffeomorphic to $\mathbb{T}^3$, or 
fibers over $\mathbb{S}^1$ with fiber $\mathbb{T}^2$ and periodic 
monodromy $A \in \mathrm{SL}(2,\mathbb{Z})$. By 
\cite[Theorem~4.1]{gnandi2025}, the finite-order elements 
of $\mathrm{SL}(2,\mathbb{Z})$ are precisely those conjugate in 
$\mathrm{GL}(2,\mathbb{Z})$ to the four matrices listed above, 
which completes the proof.
\end{proof}

\begin{remark}\label{rem:nonorientable}
Theorem~\ref{thm:main} extends to the non-orientable case. Indeed,
let $(M^d, g, \nabla)$ be a compact, possibly non-orientable Hessian
manifold, and let $\pi \colon \hat{M}^d \to M^d$ be its orientation
double cover. Then $(\hat{M}^d, \pi^*g, \pi^*\nabla)$ is a compact
oriented Hessian manifold to which Theorem~\ref{thm:main} applies,
yielding that $\pi^*g$ is flat on $\hat{M}^d$. Since $\pi$ is a 
local isometry, $g$ is flat on $M^d$ as well. Moreover, using 
\cite[Theorem~(4),~p.~716]{brasher1969}
$$b_s(\hat{M}) = b_s(M) + b_{d-s}(M),$$
applied with $s = 1$, we obtain
$$b_1(M) \;\leq\; b_1(\hat{M}) \;\leq\; d.$$
\end{remark}
We can now state the following corollary.
\begin{corollary}
Let $(M^3, g,\nabla)$ be a compact non-orientable $3$-dimensional Hessian manifold. Suppose that the orientable double cover $(\hat{M}^3,\pi^*g, \pi^*\nabla)$ defined in Remark~\ref{rem:nonorientable} satisfies the assumptions of Theorem~\ref{thm:main}. Then $\hat{M}^3$ is diffeomorphic to one of the manifolds listed in Proposition~\ref{pro:dim3}, and $M^3$ is diffeomorphic to one of the four flat manifolds $\mathfrak{B}_1$, $\mathfrak{B}_2$, $\mathfrak{B}_3$, $\mathfrak{B}_4$ of \cite[pp.~122--123]{Wolf1974}, with $b_{1}(M^3)\in \{1, 2\}$.
\end{corollary}

\begin{proof}
Let $\pi \colon \hat{M}^3 \to M^3$ be the orientable double cover of Remark~\ref{rem:nonorientable}. Since $\pi$ is a local isometry, $(\hat{M}^3, \pi^*g, \pi^*\nabla)$ is a compact orientable Hessian manifold; by assumption it satisfies the hypotheses of Theorem~\ref{thm:main}, so $\hat{M}^3$ is diffeomorphic to one of the manifolds in Proposition~\ref{pro:dim3}. Since $\pi$ is a local isometry, $(M^3, g)$ is flat. The classification of compact flat non-orientable $3$-manifolds \cite{Wolf1974} then forces $M^3$ to be diffeomorphic to one of $\mathfrak{B}_1, \mathfrak{B}_2, \mathfrak{B}_3, \mathfrak{B}_4$. In each case the first homology group $H_1(M^3;\mathbb{Z})$ is isomorphic to one of
\[
    \mathbb{Z}^2 \oplus \mathbb{Z}_2, \quad \mathbb{Z}^2, \quad \mathbb{Z} \oplus \mathbb{Z}_2^2, \quad \mathbb{Z} \oplus \mathbb{Z}_4,
\]
from which we read off $b_1(M^3) \in \{1,2\}$.
\end{proof}

\begin{corollary}
Every closed Hessian $3$-manifold $(M^3, g, \nabla)$ of Koszul type (hyperbolic) is a 
Seifert fibered space with vanishing Euler number whose base orbifold is hyperbolic .
\end{corollary}

\begin{proof}
By \cite[Theorem~4.1]{gnandi2025}, every closed oriented Hessian $3$-manifold
is diffeomorphic to one of the following: the Hantzsche--Wendt manifold; the
$3$-torus $\mathbb{T}^3$; a torus bundle $\mathbb{T}^2_A$ with monodromy
\[
A \in \left\{
    \begin{pmatrix} 0 & 1 \\ -1 & 0 \end{pmatrix},\quad
    \begin{pmatrix} -1 & 0 \\ 0 & -1 \end{pmatrix},\quad
    \begin{pmatrix} 0 & -1 \\ 1 & 1 \end{pmatrix},\quad
    \begin{pmatrix} -1 & -1 \\ 1 & 0 \end{pmatrix}
\right\};
\]
or a quotient $(\mathbb{H}^2 \times \mathbb{R})/\Gamma$, where $\Gamma \subset
\mathrm{Isom}(\mathbb{H}^2 \times \mathbb{R})$ is a discrete cocompact subgroup
acting freely. We eliminate the first three families. By
\cite[Corollaire]{koszul1965}, every Hessian manifold of Koszul type
satisfies $b_1(M^3) \geq 1$. Since the Hantzsche--Wendt manifold has first
Betti number $b_1 = 0$, it is excluded. The torus $\mathbb{T}^3$ and the torus
bundles $\mathbb{T}^2_A$ are excluded by Corollary~\ref{cor:no-nonneg-ricci}
together with Proposition~\ref{pro:dim3}. The only remaining case is the
$\mathbb{H}^2 \times \mathbb{R}$ geometry in the sense of Thurston, and every
closed $3$-manifold admitting this geometry is a Seifert fibered space with vanishing Euler number whose base orbifold is hyperbolic, which completes the proof.
\end{proof}

\section*{Acknowledgments}	
  
The authors are grateful to Aziz El Kacimi Alaoui (Universit\'e Polytechnique Hauts-de-France, Laboratoire DMATHS-C\'ERAMATHS) for invaluable discussions on codimension foliation and for his support. 
  
\bibliographystyle{plain}
\bibliography{main}
\end{document}